\documentclass[a4paper]{amsart}
\usepackage{times}
\usepackage{cite}
\usepackage{amsthm,amsmath,amssymb,amsfonts}
\usepackage{algorithmic}
\usepackage{graphicx}
\usepackage{textcomp}
\usepackage{mathtools}

\usepackage{ stmaryrd }
\usepackage{mfl}

\usepackage{bussproofs}

\usepackage{ mathrsfs }

\usepackage{relsize}

\usepackage{amsmath}
\usepackage{placeins}
\usepackage{array}

\usepackage{hyperref}
\usepackage{bussproofs}

\newcommand{\RULE}[2]{\cfrac{#1}{\raisebox{-1mm}{\ensuremath{#2}}}}

\usepackage{tikz}
\usetikzlibrary{arrows,chains,matrix,positioning,scopes}
\def\BibTeX{{\rm B\kern-.05em{\sc i\kern-.025em b}\kern-.08em
    T\kern-.1667em\lower.7ex\hbox{E}\kern-.125emX}}

\makeatletter
\tikzset{join/.code=\tikzset{after node path={%
\ifx\tikzchainprevious\pgfutil@empty\else(\tikzchainprevious)%
edge[every join]#1(\tikzchaincurrent)\fi}}}
\makeatother

\tikzset{>=stealth',every on chain/.append style={join},
         every join/.style={->}}
\tikzstyle{labeled}=[execute at begin node=$\scriptstyle,
   execute at end node=$]

\newtheorem{Def}{Definition}
\newtheorem{Thm}{Theorem}
\newtheorem{Rmk}[Thm]{Remark}
\newtheorem{Exm}[Thm]{Example}

\newtheorem{Lem}[Thm]{Lemma}

\usepackage{color,dsfont}

\theoremstyle{plain}

\theoremstyle{definition}
\usepackage[]{graphicx}

\title{A parametrised axiomatization  for a large number of restricted second-order logics}

 \author{Guillermo Badia }
\address{University of Queensland, Brisbane, Australia}
\email{g.badia@uq.edu.au}
 \author{John L. Bell}
 \address{University of Western Ontario, London, Canada}
 \email{jbell@uwo.ca}
 
 \usepackage{fancyhdr}

\begin{document}


\begin{abstract}
By limiting the range of the predicate variables in a second-order language one may obtain restricted versions of second-order logic such as weak second-order logic or definable subset logic. In this note we provide an infinitary strongly complete axiomatization for several systems of this kind having the range of the predicate variables as a parameter. The completeness argument uses  simple techniques from the theory of Boolean algebras.

\medskip

\medskip

 \end{abstract}

\dedicatory{This article is dedicated to our friend John N. Crossley on the occasion of his 86th birthday.}
\maketitle

\section{Introduction}

Second-order logic famously extends first-order (or ``elementary") logic by allowing for the possibility of  quantification not just over elements of a given universe of discourse but over properties or relations in said universe \cite{Bell2, va}. There is, however, a choice to be made as to what subsets or relations of the universe we should be allowed to quantify over. Restricted second-order logics prevent quantification over arbitrary subsets and may circumscribe attention, for example in the case of so called weak second-order logic, to finite subsets or to subsets definable without parameters in definable subset logic.   Once one makes such choice, though, a natural question to ask is whether the resulting logic can be shown to be complete for some formal system.

In this sense one might ask for two kinds of completeness, \emph{weak} (every validity is provable) or \emph{strong} (every valid inference from a set of premises can be turned into a deduction). In the case of weak second-order logic, L\'opez-Escobar provided  the first weakly complete infinitary axiomatization (in both Gentzen and Hilbert calculi presentations)  to appear in print  \cite{LP}. \red{Tarski had introduced weak second-order logic in \cite{AT}, and its axiomatization problem had been proposed to L\'opez-Escobar by Mostowski according to a remark in  \cite{LP}.} L\'opez-Escobar employed  a tableaux argument   in his central result and formulated the system to handle finite sequences of objects rather than sets.  However, the second author  had already independently obtained this result by algebraic methods in his D. Phil. thesis \cite{Bell} under the supervision of John N. Crossley. This latter approach, as we shall see here, has the advantage of  being an instance of an abstract argument applying to a plethora of restricted second-order languages.

The purpose of this note is to formulate  the abstract argument  that covers both completeness proofs in \cite{Bell} (for weak second-order logic and definable subset logic) as well as several other cases (e.g. for the logics introduced in \cite{lin}).  Essentially, if a restricted second-order logic takes as the range of its second-order variables those subsets of a domain that are definable (possibly with parameters) by some countable set $\Theta$  of first-order formulas, our completeness argument will work for that system. \red{The set $\Theta$ serves thus as a parameter for the axiomatization that we provide. Neither the system in  \cite{Bell} nor the one in  \cite{LP} were parametrised in this manner, rather they were built for particular cases of the logics discussed here (namely, weak second-order logic and definable subset logic in a countable vocabulary). Thus what we introduce here is a genuine generalization of previous work.}  The algebraic techniques employed in the present article are well-known from the theory of Boolean algebras (the reader may consult \cite{Bell&Slomson, Bell3} for an introductory treatment).

\section{The axiomatization}

Let $\Theta$ be a countable set of first-order formulas. \red{The restriction to a countable set at this point has to do with the infinitary rule (R3) in the proof system that we will present below.}  By the model-theoretic language $\mathcal{L}^{2}_{\Theta}$ we mean a second-order language (with the primitives $\neg, \wedge$ and $\forall$ -and other symbols defined as usual-, \red{as well as a countable list of first-order variables $x, y, z, \dots$}) where the (countably many) second-order variables $V^{l+1}_m$  ($m, l=0, 1,2, \dots$) are meant to range over the relations of any finite arity (given by $l+1$) of a given domain that are definable (possibly with parameters) by formulas in $\Theta$.  Occasionally, we will drop the superscripts to ease the notation when they are clear from the context. We do not impose any restriction on the finite number of free variables each member of $\Theta$ should have other than the fact that it should be at least one. 

More precisely, if $\mathfrak{A}$ is a structure \red{and $A$ its domain}, let $K_{\Theta}^A $ be the collection of all $ B \subseteq A^k$ for all $k$ such that there is a formula $\theta(\overline{x}, \overline{y}) \in \Theta$ with $\overline{x}$ being a sequence of first-order variables of length $k$ (intuitively this is the arity of the relation defined by $\theta(\overline{x}, \overline{y})$), and
$$B = \{\overline{d} \mid \mathfrak{A} \models \theta[\overline{d}, \overline{e}]\} \ \text{for some sequence $\overline{e}$ of elements of $A$}.$$
 \red{In other words, the range of every second
order variable $V^{l+1}_n$ of arity $l+1$ over a structure $\mathfrak{A}$ is the set of all such defined $l+1$-ary relations
$B$.} In this case, we call $(\mathfrak{A}, K_{\Theta}^A)$  a \emph{standard structure} for $\mathcal{L}^{2}_{\Theta}$.  \red{If for some
natural number $k$, $\Theta$ contains no formula in $k$ variables, the
range of every second order variable 
$V^k_n$ of arity $k$ in a structure $(\mathfrak{A}, K_{\Theta}^A)$ is assumed to be empty.}

\begin{Exm}\label{wso}\emph{
Weak second order logic is the case where $\Theta = \{\bigvee_{i\leq n} x = y_i \mid n  \in \omega\}$ \red{since in this system the second-order variables range over the finite subsets of the domain while the formulas in $\Theta$ define with parameters all such finite subsets in  any give structure $\mathfrak{A}$}.}
\end{Exm}

\begin{Exm}\label{dsl}\emph{
Definable subset logic  \red{in a countable vocabulary} \cite{Bell} is the case where $\Theta$ is just all first-order formulas in one free variable. In other words, in definable subset logic the range of the unary second-order variables is the collection of all subsets definable by a formula in one free variable in a a structure.}

\end{Exm}

\begin{Exm} \emph{The logic of elementarily definable  (respectively elementarily definable with parameters) relations \cite{lin} \red{in a countable vocabulary} is the case  where we take $\Theta$ to be the set of all first-order formulas (and the $B$s in the definition of $K_{\Theta}^A $  above require suitable parameters).} 
\end{Exm}

\begin{Exm}\emph{\red{Recall that the hierarchy of formulas $\exists_n$ (respectively $\forall_n$) for any $n$ is defined as those having, roughly, alternating blocks of existential (universal) and universal (existential) quantifiers, cf. \cite[47-48]{Hodges}.}
$\exists_n$ (respectively $\forall_n$)-definable subset logic  \red{in a countable vocabulary} is the case where $\Theta$ is just all $\exists_n$ (respectively $\forall_n$) first-order formulas with parameters.}

\end{Exm}

Now enumerate the members of $\Theta$ as $\theta^{l+1}_n$ ($n, l = 0, 1, 2, \dots$) where $l+1$ indicates the arity of the relation determined by $\theta^{l+1}_n$ in the sense of the number of free variables of the formula that are not parameters. \red{Given a formula $\f$, we will write $\f(V_m/V_n)$ for the result of replacing $V_m$ by $V_n$ in the formula $\f$. Furthermore, we write $ \f^{l+1, m}_n$ for the result of replacing any expression $V_m^{l+1}(\overline{x})$ (that is not bound by a quantifier) in $\f$ by the formula $\theta^{l+1}_n(\overline{x}, \overline{y}) $ from our enumeration  and prefixing the resulting formula with the string of quantifiers $\All{\overline{y}}$.} The formal system for $\mathcal{L}^{2}_{\Theta}$ will contain any complete set of axioms for first-order logic in addition to the following:

\begin{itemize}
\item[] {\sc Axiom schemata}

\item[] (Comprehension)
\item[(A1)] $\All{\overline{y}}\Exi{V}\All{\overline{x}} (V(\overline{x}) \leftrightarrow \theta(\overline{x}, \overline{y}) )$ \,  \, [for each $\theta(\overline{x}, \overline{y}) \in \Theta$ ] 

\item[] (Extensionality)
\item[(A2)] $ \All{V_m, V_n}( \All{\overline{x}} (V_m(\overline{x}) \leftrightarrow V_n(\overline{x})  ) \leftrightarrow V_m = V_n )$ \,  \, [when $V_m, V_n$ have the same arity]

\item[] (Leibniz's Law)

\item[(A3)] $\All{V_m, V_n}(V_m  = V_n \rightarrow (\f \rightarrow \f') )$ \, \, where $ \f'$ results from $\f$ by replacing some of the occurrences of $V_m$ by $V_n$ and both second-order variables have the same arity.

\item[] (Quantifier Principles)

\item[(A4)] $\All{V_m} \f \rightarrow \f(V_m/V_n) $ \, \, where $V_n$ is 
  free for $V_m$ in $\f$ and both second-order variables have the same arity.

\item[(A5)] $\All{V_m} (\f \rightarrow \p) \rightarrow (\f \rightarrow \All{V_m}\p) $ \, \, where $V_m$ is not free in $\f$.

\item[]  For each $m, n \in \omega$,
 \item[(A6)]    $\All{V_m^{l+1}} \f \rightarrow \f^{l+1, m}_n$.

\end{itemize}

\begin{itemize}
\item[] {\sc Rules}
\item[(R1)]  \emph{Modus Ponens}:  
\[  \RULE{\f \rightarrow \p \qquad \f}{\p}\]
\item[(R2)] \emph{Generalization}: \[  \RULE{\f}{\forall V \f} \qquad  \RULE{\f}{\forall x \f}\]
\item[(R3)]  \emph{Infinitary rule}:  
\[  \RULE{\psi \rightarrow \f^{l+1,m}_0, \psi \rightarrow \f^{l+1,m}_1, \psi \rightarrow \f^{l+1,m}_2, \dots \qquad }{\psi \rightarrow \All{V^{l+1}_m}\f }\]
\end{itemize}

\red{Intuitively, (R3) is analogous to the central rule of inference in $\omega$-logic which plays a similar role in that context as ours here \cite{orey}. Observe that if $\Theta$ is allowed to be uncountable, there is no clear way of formulating (R3) in a sound way, as the rule requires a countable list of premises and there are too many possible values for the second-order variable $V^{l+1}_m$. For example, if our vocabulary contains a list $\{ c_\alpha\mid \alpha \in \omega_1\}$ of constants and a unary predicate $P$, and we let $\Theta = \{ y = c_\alpha\mid \alpha \in \omega_1\}$, then clearly $\{ \All{y} (y = c_\alpha \rightarrow Py) \mid \alpha \in \omega_1\} \vDash \All{X}\All{y} (Xy \rightarrow Py)$ but there is no countable subset of $\{ \All{y} (y = c_\alpha \rightarrow Py) \mid \alpha \in \omega_1\}$ from which $\All{X}\All{y} (Xy \rightarrow Py)$  follows.}

A \emph{deduction} of a formula $\f$ from a set of premises $\Sigma$ is simply a countable sequence of formulas such that $\f$ is the last member of the sequence and every element in the sequence is either a member of $\Sigma$, an axiom or it comes from previous members of the sequence by means of one of (R1)-(R3). It is an easy exercise to see that this system is sound with respect to the standard models $(\mathfrak{A}, K_{\Theta}^A)$.  In the remainder of this article we will focus on establishing the converse, namely that the system is also strongly complete with countable sets of premises.

\red{We need to say a few words here on how our models are related to Henkin's so-called `general models' \cite{henkin}. As it is well-known, second-order logic with the `standard semantics' where the second-order variables range over arbitrary subsets of a suitable Cartesian product of the domain of the models   is incomplete for any finitary axiomatization for G\"odelian reasons. Henkin \cite{henkin} famously provides a way to formulate the semantics of second-order logic for which completeness can be regained. The trick is to allow for the range of values of the second-order variables to change from one model to another, rather than be always the power set of the appropriate Cartesian product of the domain in every model, modulo some closure properties on these sets of values that guarantee every variable can be interpreted.  Observe that in our approach we have  restricted the set of possible values of the second-order variables but we have done so uniformly across all models, so we are by no means allowing as much freedom as Henkin does. In general, such lack of freedom breaks down any hope of a completeness theorem for a finitary axiomatization  but if our vocabulary is rich enough sometimes such completeness can be regained by the methods of Henkin.}
 Moreover, the limitations on axiomatizability for the systems  $\mathcal{L}^{2}_{\Theta}$  will greatly depend  on the complexity of the $\Theta$ chosen and on what kind of completeness we are interested in as we will see in the following remarks. 

\medskip

\paragraph{\em Weak second-order logic} The set of validities of Example \ref{wso} is well-known not to be  axiomatizable by a finitary system \cite{Montague}. To see this, observe that the standard model of arithmetic $(\omega, +, \cdot, S,  0)$ is axiomatizable by the conjunction of the \red{(finitely many) axioms of Robinson arithmetic (which can be formalized in first-order logic already)} and the statement that every element has only finitely many predecessors \red{(which can be written as $\All{x}\Exi{V}\All{y}(y<x \rightarrow V(y))$)} \cite[p. 488]{monk}. Thus if the validities of weak second-order logic would be recursively enumerable, true arithmetic would be as well, contradicting Tarski's theorem on the undefinability of truth.

\medskip 

\paragraph{\em Definable subset  logic} In the case of Example \ref{dsl}, as it is reported towards the end of \cite{lin}, Peter Aczel showed that  the finitary system without R3 or A6 axiomatizes the valid formulas if the vocabulary of the logic is allowed to have  denumerably many predicate constants (i.e.  it is weakly complete). This is done by adding a new countable set of predicate constants and building a Henkin theory where each second-order existential theorem of the theory is witnessed by one of the predicate constants. The argument proceeds from here in the usual style of Henkin. 

On the other hand, no strong completeness result is possible by employing a finitary axiomatization as it would imply compactness, which fails for definable subset logic. The latter can be seen by adding a new individual constant $c$ to the vocabulary of arithmetic and considering the theory $\Delta$ which results by adding to $\text{Th}^{DSL}(\omega)$ (the set of all definable subset sentences true in  $(\omega, +, \cdot, S,  0)$) the following sentences:
$$\{\neg \f(c) \mid  \f(x) \ \text{is a first-order formula in vocabulary  $\{+, \cdot, S, 0\}$ defining an element of} \ (\omega, +, \cdot, S, 0)\}. $$
Then every finite $\Delta' \subset \Delta$ has a model, namely $(\omega, +, \cdot, S, 0)$ can be expanded into a model of $\Delta'$. However, we can see that $\Delta$ itself cannot have a model $\mathfrak{A}$. Otherwise its reduct,   $\mathfrak{A}'$,  to the vocabulary $\{+, \cdot, S, 0\}$, being a model of $\text{Th}^{DSL}(\omega)$,  would be an elementary extension in the sense of definable subset logic of $(\omega, +, \cdot, S, 0)$. Since  $ (\omega, +, \cdot, S, 0)\models \All{x}\Exi{V}\All{y}(Vy \leftrightarrow x=y)$, we have that $\mathfrak{A}' \models \All{x}\Exi{V}\All{y}(Vy \leftrightarrow x=y)$, but then $c^{\mathfrak{A}}$ (the denotation of $c$ in $\mathfrak{A}$) must be definable in $\mathfrak{A}'$ by a formula $\f$ of first-order arithmetic  and hence, given that $\mathfrak{A}' \models \Exi{x} \All{y} (\f(y) \leftrightarrow x=y)$, we have that $(\omega, +, \cdot, S, 0) \models \Exi{x} \All{y} (\f(y) \leftrightarrow x=y)$, which contradicts the fact that $\mathfrak{A} \models \neg \f (c)$ by definition of $\Delta$.

\medskip 

\paragraph{\em Elementarily definable relation logic}  Lindstr\"om has shown in \cite{lin} that the set of validities of this logic in a sufficiently rich (in the sense of containing the vocabulary of arithmetic) \emph{finite} vocabulary is not just not recursively enumerable but $\Pi^1_1$-complete. By a rather clever argument he shows that one can implicitly define the standard model of arithmetic $(\omega, +, \cdot, S,  0)$ by a sentence in this logic. In contrast, when the vocabulary is allowed to be infinite, the same argument by Aczel mentioned above works to axiomatize the validities in this context. However, once more since compactness is lost, no finitary axiomatization could yield strong completeness.

\section{The (Strong) Completeness theorem}

In this section we will present the completeness argument that generalises the concrete instances in \cite{Bell}. We start  by recalling and introducing some notions about Boolean algebras in the next definitions (we use the notation from \cite[Chp. 4]{Bell3}) \red{where $\wedge$ and  $*$   are the meet  and  complement  operations, respectively.}
 \begin{Def}
 Let $\mathcal{B}$ be a Boolean algebra. A subset $U$  of $\mathcal{B}$ is an \emph{ultrafilter} if 
 \begin{item}
 \item[(i)] $1 \in U, 0 \notin U$,
 \item[(ii)] $a, b \in U $ only if $a \wedge b \in U$,
  \item[(iii)] $a \in U$ and $a \leq b$ only if $b \in U$,
    \item[(iv)] for any $a$ from $\mathcal{B}$, either $a \in U$ or $a^* \in U$.
 \end{item}
   \end{Def}
\begin{Def}
 Let $\mathcal{B}$ be a Boolean algebra and $\mathcal{F}$ a family of subsets of $\mathcal{B}$. We will say that $\mathcal{F}$ is \emph{regular} if each member of $\mathcal{F}$ has a join and a meet in $\mathcal{B}$. 
 \end{Def}

 \begin{Def}
 Let $\mathcal{B}$ be a Boolean algebra and $\mathcal{F}$ a regular family of subsets of $\mathcal{B}$.
 If $U$ is an ultrafilter in  $\mathcal{B}$, we will call it \emph{$\mathcal{F}$-compatible} when for each $S \in \mathcal{F}$ the following holds:
\begin{itemize}
\item[(i)] $\bigvee S \in U $ iff $S \cap U \neq \emptyset$,
\item[(ii)] $\bigwedge S \in U $ iff $S \subseteq U $. 
\end{itemize}
\end{Def}
\medskip

Given a countable set $\Sigma$ of  $\mathcal{L}^{2}_{\Theta}$-sentences, we can build the Lindenbaum algebra $\mathscr{L}(\Sigma)$ as the algebra of equivalence classes  $\llbracket\f\rrbracket$ of formulas of $\mathcal{L}^{2}_{\Theta}$ under the equivalence relation $\Sigma \vdash \f \leftrightarrow \p$ and the quotient operations derived from the connectives. The lattice ordering on the algebra $\mathscr{L}(\Sigma)$ is simply $\llbracket\f\rrbracket \leq \llbracket\p\rrbracket$ iff $\Sigma \vdash \f \rightarrow \p$. Naturally, $\mathscr{L}(\Sigma)$  is a Boolean algebra.

Now we can establish a fact that  will make essential use of the infinitary rule of our axiomatization and will be needed in our completeness argument:
\begin{Lem}\label{reg} Let $\mathscr{L}(\Sigma)$ be the Lindenbaum algebra of a countable set $\Sigma$ of  $\mathcal{L}^{2}_{\Theta}$-sentences. Then for each formula $\f$ of $\mathcal{L}^{2}_{\Theta}$,
\begin{itemize}
\item[(i)] $\llbracket\All{x_m} \f\rrbracket  = \bigwedge_{n \in \omega} \llbracket\f(x_m/x_n)\rrbracket$,
\item[(ii)] $\llbracket \Exi{x_m} \f\rrbracket = \bigvee_{n \in \omega} \llbracket\f(x_m/x_n)\rrbracket$,
\item[(iii)] $\llbracket\All{V_m} \f\rrbracket  = \bigwedge_{n \in \omega} \llbracket\f(V_m/V_n)\rrbracket$,
\item[(iv)] $\llbracket\Exi{V_m} \f\rrbracket  = \bigvee_{n\in \omega} \llbracket\f(V_m/V_n)\rrbracket$,
\item[(v)] $\llbracket\All{V_m} \f\rrbracket  = \bigwedge_{n\in \omega} \llbracket\f^m_n\rrbracket$,
\item[(vi)] $\llbracket\Exi{V_m} \f\rrbracket  = \bigvee_{n\in \omega} \llbracket\f^m_n\rrbracket$.
\end{itemize}

\end{Lem}
\begin{proof}
This can be seen by relatively familiar arguments (see \cite{Bell&Slomson}) using (Quantifier Principles) and R2-R3. In particular, R3 is needed for the proof of (v) and, dually, (vi).
\end{proof}

Let $\Sigma \cup \{\f\}$ now be a countable set of $\mathcal{L}^{2}_{\Theta}$-sentences such that $\Sigma \nvdash \{\f\}$. Take the family $\mathcal{F}$  of subsets of the Lindenbaum algebra $\mathscr{L}(\Sigma)$ consisting of all sets of the form 
$$\{\llbracket \p (x_n/x_m) \rrbracket \mid m \in \omega\}, \{\llbracket \p (V_n/V_m)\rrbracket \mid m \in \omega\}, \{\llbracket \p^n_m \rrbracket \mid m\in \omega\}$$
for arbitrary formulas $\p$ of $\mathcal{L}^{2}_{\Theta}$. This is a countable regular family by Lemma \ref{reg}, and by using the Rasiowa-Sikorski lemma \cite{RS}, we may obtain for any non-unit element of $\mathscr{L}(\Sigma)$, and hence in particular for $\llbracket \f\rrbracket $, an $\mathcal{F}$-compatible ultrafilter $U$ not containing that element,  so in this case, $\llbracket \f\rrbracket \notin U$. For each first-order variable $x$ we denote by $\widehat{x}$ the equivalence class $\{y \mid \llbracket x=y\rrbracket \in U\}$ and by $F_m$ the set $\{\red{\langle \widehat{x}_0, \dots, \widehat{x}_l \rangle} \mid \llbracket V_m(x_0, \dots, x_l )\rrbracket \in U\}$ when $l$ is the arity of the relation variable $V_m$.

\begin{Lem}\label{use1} For each second-order variable $V_m$, $\overline{y}$ a sequence of first-order variables of the same length as the arity of $V_m$, there is a finite set $\{x_0, \dots, x_k\}$ of first-order variables and a natural number $n$ such that $$\llbracket\All{\overline{y}}(V_m (\overline{y}) \leftrightarrow \theta_n(\overline{y}, x_0, \dots, x_k))\rrbracket \in U.$$ 

\end{Lem}

\begin{proof} Suppose for a contradiction that for each finite set $\{x_0, \dots, x_k\}$ of first-order variables and  natural number $n$, $$\llbracket\All{\overline{y}}(V_m (\overline{y}) \leftrightarrow \theta_n(\overline{y}, x_0, \dots, x_k))\rrbracket \notin U.$$ 
Since our axiomatization contains a complete set of axioms for first-order logic, for any first-order formula $\f$,  $\llbracket\neg \All{x} \f\rrbracket\leq \llbracket  \Exi{x}  \neg\f\rrbracket$. Consequently, given that $U$ is an ultrafilter in a Boolean algebra, for each set $\{x_0, \dots, x_k\}$  and  natural number $n$,
$$\llbracket\Exi{\overline{y}}\neg (V_m (\overline{y}) \leftrightarrow \theta_n(\overline{y}, x_0, \dots, x_k))\rrbracket \in U,$$ 
and, once more, since we have a complete set of axioms for first-order logic,
$$\llbracket \All{\overline{y}}(\theta_n(\overline{y}, x_0, \dots, x_k) \rightarrow V_m (\overline{y})) \rightarrow \Exi{\overline{y}} (V_m (\overline{y}) \wedge  \neg \theta_n(\overline{y}, x_0, \dots, x_k))\rrbracket\in U.$$ 
Since $U$ is $\mathcal{F}$-compatible \red{and using Lemma \ref{reg} (i)}, for each natural number $n$ and $\{x_0, \dots, x_k\}$, $$\llbracket\All{x_0, \dots, x_k}( \All{\overline{y}}(\theta_n(\overline{y}, x_0, \dots, x_k) \rightarrow V_m (y)) \rightarrow \Exi{\overline{y}} (V_m (\overline{y}) \wedge  \neg \theta_n(\overline{y}, x_0, \dots, x_k)))\rrbracket\in U.$$ 
By $\mathcal{F}$-compatibility \red{and  Lemma \ref{reg} (v)}, 
$$\llbracket\All{V_p} (\All{\overline{y}}(V_p(\overline{y}) \rightarrow V_m (\overline{y})) \rightarrow \Exi{\overline{y}} (V_m (\overline{y}) \wedge  \neg V_p(\overline{y})))\rrbracket\in U.$$ 
Thus, $$0=\llbracket\All{y}(V_m(\overline{y}) \rightarrow V_m (\overline{y})) \rightarrow \Exi{\overline{y}} (V_m (\overline{y}) \wedge  \neg V_m(\overline{y}))\rrbracket \in U,$$ which is a contradiction as $U$ is an ultrafilter.

\end{proof}

\begin{Lem}\label{use} For each finite set $\{x_0, \dots, x_k\}$ of first-order variables and  natural number $n$, there is a $V_m$ such that $$\llbracket\All{\overline{y}}(V_m (\overline{y}) \leftrightarrow \theta_n(\overline{y}, x_0, \dots, x_k))\rrbracket \in U$$
where $\overline{y}$ is sequence of the same length as the arity of $V_m$.

\end{Lem}

\begin{proof} By (A1), we have that $$\llbracket\All{x_0, \dots, x_k}\Exi{V}\All{\overline{y}}(V (\overline{y}) \leftrightarrow \theta_n(\overline{y}, x_0, \dots, x_k))\rrbracket\in U.$$ Thus, by $\mathcal{F}$-compatibility of $U$ \red{and  Lemma \ref{reg} (iv)}, we can infer that for some $m$, $$\llbracket\All{\overline{y}}(V_m (\overline{y}) \leftrightarrow \theta_n(\overline{y}, x_0, \dots, x_k))\rrbracket \in U.$$

\end{proof}

With Lemmas \ref{use1} and \ref{use} in hand, we are now in a position to prove a model existence theorem:

\begin{Thm}[Model Existence] Let $\Sigma \cup \{\f\}$ be a countable set of $\mathcal{L}^{2}_{\Theta}$-sentences. Then if $\Sigma \nvdash \f$ there is a model $(\mathfrak{A}, K_\Theta^A) $ where $(\mathfrak{A}, K_\Theta^A)\models \Sigma $ and $(\mathfrak{A}, K_\Theta^A)\not \models \f $.

\end{Thm}

\begin{proof} \red{To simplify the presentation suppose that our vocabulary consists of the predicates $P_i$ ($i \in \omega$) each with its own arity.} We take the $\mathcal{F}$-compatible ultrafilter $U$ obtained by the Rasiowa-Sikorski lemma for $\mathscr{L}(\Sigma)$  described above such that $\llbracket \f\rrbracket \notin U$. We build a canonical model as follows. Let the domain be the set $I$ containing all equivalence classes $\widehat{x}$ for each first-order variable $x $ from our countable \red{supply of such variables $x, y, z, \dots$} Consider now the structure $\mathfrak{A} = (I, R_0, R_2, \dots, R_n, \dots)$ where each $R_i$ is an interpretation for the predicate constant $P_i$ defined as follows:
$$ R_i = \{ \tuple{\widehat{x}_0, \dots, \widehat{x}_{p(i)}}\mid \llbracket P_i (x_0, \dots, x_{p(i)})\rrbracket \in U\}$$
where $p(i)$ is simply the arity of  $P_i$.

For an $\mathcal{L}^{2}_{\Theta}$-formula $\p(V_1, \dots, V_n; x_0, \dots, x_m)$,  by interpreting each $V_i$ as $F_i$ \red{(that recall we have defined as $\{\widehat{x}_0, \dots, \widehat{x}_l \mid \llbracket V_i(x_0, \dots, x_l )\rrbracket \in U\}$)}, we can see by induction on the complexity of $\p$ that $$\mathfrak{A} \models \p[F_1, \dots, F_n; \widehat{x}_0, \dots, \widehat{x}_m] \ \text{iff} \ \llbracket \p \rrbracket \in U.$$
Our next step is to show that $K_\Theta^A = \{F_i \mid i\in \omega\}$. Let us show first that $K_\Theta^A \subseteq \{F_i \mid  i\in \omega\}$. Suppose then that $B  \in K_\Theta^A$, so \red{$B \subseteq I^k$} and there is $ \theta(\overline{x}, \overline{y}) \in \Theta \ \text{s.t.} \ B = \{\overline{d} \mid \mathfrak{A} \models \theta[\overline{d}, \overline{e}]\}$ for some sequence $\overline{e}$ of elements of $I$. Say that $\overline{e}$ is the sequence $\widehat{x}_0, \dots, \widehat{x}_l$. By Lemma \ref{use}, there is a $V_m$ such that $$\llbracket\All{\overline{y}}(V_m (\overline{y}) \leftrightarrow \theta_n(\overline{y}, x_0, \dots, x_l))\rrbracket \in U.$$ 
Hence, $F_m = B$. In a similar manner, using Lemma \ref{use1}, we can show $  \{F_i \mid i>0\} \subseteq K_\Theta^A$.
Then by the choice of $U$, $(\mathfrak{A}, K_\Theta^A) \models \Sigma$ and  $(\mathfrak{A}, K_\Theta^A) \not \models \f$.
\end{proof}

\begin{Thm}[Strong Completeness]\label{com}
Let $\Sigma$ be a countable set of $\mathcal{L}^{2}_{\Theta}$-sentences. Then for any $\mathcal{L}^{2}_{\Theta}$-sentence $\f$, we have that $\Sigma \vDash \f$ implies that $\Sigma \vdash \f$.
\end{Thm}
\begin{proof}
Suppose that $\Sigma \nvdash \f$. Then, by the Model Existence theorem, we have a structure $(\mathfrak{A}, K_\Theta) $ such that $(\mathfrak{A}, K_\Theta)\models \Sigma $ and $(\mathfrak{A}, K_\Theta)\not \models \f $.
\end{proof}

\red{\begin{Rmk}\emph{From Theorem \ref{com}, we can immediately get a Deduction Theorem: \emph{Let $\Sigma \cup \{\f, \p\}$ be a countable set of $\mathcal{L}^{2}_{\Theta}$-sentences. Then if $\Sigma, \f \vdash \p$, it follows that  $\Sigma  \vdash  \f \rightarrow \p$.} To see this, notice that 
$\Sigma, \f \vdash \p$ implies that $\Sigma, \f \vDash \p$ by soundness, whence $\Sigma \vDash \f \rightarrow  \p$. So by Theorem \ref{com} we must have  that $\Sigma  \vdash  \f \rightarrow \p$ as desired.}\footnote{This result can also be obtained directly by the familiar syntactic argument in most textbooks (e.g. \cite{Bell3, Mendelson}). In this case, one proceeds by transfinite induction and the only interesting case is  if $\p = \chi \rightarrow \All{V}\sigma$ has been obtained by an application of R3 from  $\chi \rightarrow \sigma^m_0, \chi \rightarrow \sigma^m_1, \chi \rightarrow \sigma^m_2, \dots$ By inductive hypothesis, we have deductions of $\f \rightarrow (\chi \rightarrow \sigma^m_0), \f \rightarrow (\chi \rightarrow \sigma^m_1), \f \rightarrow (\chi \rightarrow \sigma^m_2), \dots$ By propositional logic, we can obtain deductions of  $\f \wedge \chi \rightarrow \sigma^m_0, \f \wedge \chi \rightarrow \sigma^m_1, \f \wedge \chi \rightarrow \sigma^m_2, \dots$ Thus, applying R3, we get $\f \wedge \chi \rightarrow \All{V}\sigma$, and by propositional logic once more, we have $\f \rightarrow \p$.}
\end{Rmk}}

\begin{Rmk}
\emph{Observe that the present method also works, \emph{mutatis mutandis},  when  we drop identity from the language of $\mathcal{L}^{2}_{\Theta}$. In some of these logics (such as in weak second order logic) this might have no effect because identity is definable. However, in a language without identity, elementarily definable with parameters relation logic, for example, is  not  as expressive as its counterpart with identity.\footnote{The techniques of \cite{Casa} which were used in \cite{Badia} already to study extensions of first-order logic without identity, could be adapted to show this. For example, one could show that \emph{reduced structures} (i.e. the quotient structures by the Leibniz congruence identifying all elements satisfying the  same identity-free formulas with parameters from a given model)  are axiomatized by the sentence $(\All{V} (V(x) \leftrightarrow V(y)) \leftrightarrow x=y)$ in the language adding identity and quantifying over elementarily definable (without identity) relations with parameters . However, in the language without identity, a model and its reduced counterpart will satisfy the same sentences of the restricted second-order logic just mentioned. }}
\end{Rmk}

\begin{Rmk}
\emph{In the present completeness argument one cannot, in general, liberalize the restriction on the countability of $\Sigma$. This is because, keeping in  mind that in weak second-order logic the standard model of arithmetic $(\omega, +, \cdot, S,  0)$ is axiomatizable by a sentence $\f$  (\red{which, as we mentioned before, is  the conjunction of the  axioms of Robinson arithmetic  and the statement that every element has only finitely many predecessors which can be written as $\All{x}\Exi{V}\All{y}(y<x \rightarrow V(y))$}), we can then build the following theory $\Delta$ in a vocabulary obtained by adding to that of $\f$ an uncountable number of new constants $\{c_\alpha\mid \alpha \in \omega_1\}$:
$$ \{\f\} \cup \{c_\alpha \neq c_\beta\mid \alpha, \beta\in \omega_1, \alpha \neq \beta\}.$$
Clearly, this theory has no model but every countable subset of it does have a model (namely an expansion of $(\omega, +, \cdot, S,  0)$). So, if we would have this more general form of strong completeness, since $\Delta \vDash \bot$, it would follow that $\Delta \vdash \bot$ and this, by our definition of a deduction, would have to be witnessed by some countable  $\Delta' \subset \Delta$, which is impossible since all of them have models. Similarly, if our vocabulary is finite but contains that of arithmetic, using Lindstr\"om's result on the implicit definability of $(\omega, +, \cdot, S,  0)$  by a sentence \cite{lin}, one could reproduce the previous argument, \emph{mutatis mutandis}, for elementarily definable relation logic as well.
}
\end{Rmk}

\section{Conclusion}

In this note we have formulated a general framework that yields completeness immediately for numerous systems of restricted second-order logic. One natural question would be to what extent the present algebraic argument is tied to a logic evaluated on the two-element Boolean algebra. For example, it is not a difficult exercise to show that what we have done here generalizes to the recent context of so called Boolean-valued second-order  logic \cite{DV}, where formulas are evaluated on an arbitrary complete Boolean algebra. Can the same be done for logics with other algebraic semantics (such as those studied in \cite{Cin})? In this regard, one may attempt to use Rauszer and Sabalski's generalization of the Rasiowa-Sikorski lemma  to distributive lattices  \cite{RSb} but the biggest obstacle to solve would be how to generalize   Lemma \ref{use1}  which makes distinctive use of properties of Boolean logic.

\section*{Acknowledgements} \red{We are very grateful to two anonymous referees for the present journal who have been extraordinarily thorough, almost above and beyond the call of duty. Their comments contributed significantly in improving our presentation of the material.}
 Badia is supported by  the Australian Research Council grant DE220100544.

\end{document}